\documentclass[12pt]{amsart}
\usepackage{amsmath}
\usepackage{amssymb}
\usepackage{amsthm}
\usepackage[final]{hyperref}
\usepackage{graphicx}
\newtheorem{thm}{Theorem}
\newtheorem{lem}[thm]{Lemma}

\newtheorem{prop}[thm]{Proposition}
\theoremstyle{definition}

\theoremstyle{definition}
\newtheorem{rem}{Remark}
\newcommand{\R}{\mathbb{R}}

\title[Analytic Solutions for NLS]{Global Analytic Solutions for the Nonlinear Schr\"odinger Equation}

\author{Daniel Oliveira da Silva}
\address{Department of Mathematics \\
Nazarbayev University \\
Qabanbay Batyr Avenue 53 \\
010000 Nur-Sultan \\
Republic of Kazakhstan}

\author{Magzhan Biyar}
\address{Department of Mathematics \\
Nazarbayev University \\
Qabanbay Batyr Avenue 53 \\
010000 Nur-Sultan \\
Republic of Kazakhstan}

\email{daniel.dasilva@nu.edu.kz}
\email{magzhan.biyarov@nu.edu.kz}


\begin{document}

\begin{abstract}
We prove the existence of global analytic solutions to the nonlinear Schr\"odinger equation in one dimension for a certain type of analytic initial data in $L^2$.
\end{abstract}

\maketitle


\section{Introduction}

The nonlinear Schr\"odinger equation is the equation
\begin{equation}\label{NLS}\begin{aligned}
i u_{t} + \Delta u = |u|^{p-1}u,
\end{aligned}\end{equation}
where $u: \mathbb{R}^{1+d} \rightarrow \mathbb{C}$.  This equation has been studied extensively for data in the Sobolev spaces $H^{s}$.  For a detailed discussion the $H^{s}$ theory for this equation, see chapter 3 of \cite{T2006} and the many references therein.  Recently, there has been much interest in developing a theory of analytic solutions to partial differential equations of all types, and many results exist in this direction.  For a brief sampling of results, see \cite{BGK2005, BGK2006, BGK2010, GK2003, GK2002, L2012, GHHP2013, HP2012, HHP2011, HHP2006, ST2015, SD2016, D2017, D2018}.

In the case of equation \eqref{NLS}, there exist several results regarding the existence of global analytic solutions.  An early result in this direction is that of Hayashi \cite{H1990}, who studied the cubic case in dimensions $d \geq 2$ for small initial data.  In the same year, Hayashi and Saitoh \cite{HS1990-2} obtained a similar result, but using milder smallness assumptions on the data.  This was later generalized by Nakamitsu \cite{N2005} to $p-1 = 2 \kappa$, where
\[
\frac{d}{2} - 1 \leq \frac{1}{\kappa} \leq \frac{d}{2},
\]
but again requiring smallness assumptions on the initial data.  For the cubic case, these smallness assumptions were later removed by Tesfahun in \cite{T2017}, who considered the problem in dimensions $d = 1,2,3$.

In the present work, we will extend these results to the one-dimensional case where $p$ can be any odd number.  In particular, we will prove the following theorem:
\begin{thm}\label{main}
Let $p$ be an odd natural number, and let $f \in L^{2}(\mathbb{R})$.  Suppose that $f$ admits a holomorphic extension $\tilde{f}$ on the set
\[
S_{\sigma_{0}} = \{ x + iy \in \mathbb{C}:\ |y| < \sigma_{0} \},
\]
and that
\[
\sup_{|y| < \sigma_{0}} \| \tilde{f}(\cdot + i y) \|_{L^{2}_{x}} < \infty.
\]
Then for any $T > 0$, the Cauchy problem
\begin{equation}\label{NLSCauchy}\begin{aligned}
& i u_{t} + \Delta u = |u|^{p-1}u, \\
& \quad u(x,0) = f(x).
\end{aligned}\end{equation}
has a unique solution $u \in C([0,T];L^{2})$.  Moreover, this solution is the restriction to the real line of a function $\tilde{u}$ which is holomorphic on the set $S_{\sigma}$, where
\[
\sigma < \min\left\{ \sigma_0, C T^{-1-\epsilon} \right\}.
\]
for some constant $C > 0$ and any $\epsilon > 0$.  Thus, the analyticity of $u$ persists for all time.
\end{thm}
\noindent For the proof, we first construct local solutions by a standard fixed-point argument.  The procedure is standard, but for completeness, it will be shown in section \ref{LWP}.  In section \ref{GWP}, we then show that the local solutions can be extended to arbitrarily large time intervals, if we allow the radius of analyticity $\sigma$ to decay.  The proof uses a bootstrap argument and an almost conserved quantity, which we control by using the parameter $\sigma$.  We begin our discussion by introducing the necessary tools in section \ref{prel}.


\section{Preliminaries}\label{prel}

An important tool in our construction of analytic solutions to \eqref{NLSCauchy} are the Gevrey spaces $G^{\sigma}(\mathbb{R})$, which are defined by the norm
\[
\| f \|_{G^{\sigma}} = \| e^{\sigma |\xi|} \hat{f}(\xi) \|_{L^{2}_{\xi}},
\]
where $\hat{f}$ denotes the spatial Fourier transform, $\langle x \rangle = (1 + |x|^{2})^{1/2}$, and $\sigma > 0$.  The importance of the Gevrey spaces comes from the following Paley-Wiener theorem, for which a proof can be found in \cite{K1976}:
\begin{thm}\label{wiener}
Let $\sigma > 0$.  Then, the following are equivalent:
\begin{enumerate}
\item $f \in G^{\sigma}(\R)$;
\item $f$ is the restriction to the real line of a function $\tilde{f}$ which is holomorphic in the strip
\[
S_{\sigma} = \{ x + i y: x,y \in \mathbb{R}, |y| < \sigma \}
\]
and satisfies
\[
\sup_{|y| < \sigma}\| \tilde{f}(x + i y) \|_{L^2_{x}} < \infty.
\]
\end{enumerate}
\end{thm}

\begin{rem}
It should be noted that there is no assumption in this theorem that the function $f$ must be real-valued.  This is important, as initial data and solutions to equation \eqref{NLS} are complex-valued.
\end{rem}
In addition to the spaces $G^{\sigma}$, we will also make use of the hybrid Gevrey-Sobolev spaces $G^{\sigma,s}(\mathbb{R})$ defined by the norm
\[
\| f \|_{G^{\sigma,s}} = \| e^{\sigma |\xi|} \langle \xi \rangle^{s} \hat{f}(\xi) \|_{L^{2}_{\xi}}.
\]
It is a simple matter to see that these spaces satisfy the embeddings
\begin{equation}\label{embedding1}
G^{\sigma',s'} \hookrightarrow G^{\sigma, s}
\end{equation}
for $\sigma \leq \sigma'$ and $s,s' \in \mathbb{R}$, which follow from the inequalities
\[
\| f \|_{G^{\sigma,s}} \lesssim \| f \|_{G^{\sigma',s'}}.
\]
Note that $G^{0,s} = H^{s}$, so that for $\sigma = 0$ the inequality becomes
\begin{equation}\label{embedding2}
\| f \|_{H^{s}} \lesssim \| f \|_{G^{\sigma',s'}}
\end{equation}
and the associated embedding is
\[
G^{\sigma',s'} \hookrightarrow H^{s}.
\]
Gevrey-Sobolev spaces also obey the following generalization to the standard alegbra property of Sobolev spaces.
\begin{lem}\label{Lemma3}
If $s > 1/2$ and $\sigma \geq 0$, then the space $G^{\sigma, s}(\R)$ is an algebra, and
\[
\| uv \|_{G^{\sigma,s}} \lesssim \| u \|_{G^{\sigma,s}} \| v \|_{G^{\sigma,s}}.
\]
\end{lem}
\begin{proof}
By definition, we have
\[
\| uv \|_{G^{\sigma,s}} = \left\| e^{\sigma|\xi|}\langle \xi \rangle^{s} \widehat{uv}(\xi) \right\|_{L^{2}_{\xi}}.
\]
Observe that
\[
\widehat{uv}(\xi) = \int_{\R} \hat{u}(\xi-\eta) \hat{v}(\eta)\ d\eta.
\]
By the triangle inequality, we also have that
\[\begin{aligned}
e^{\sigma|\xi|} & \leq e^{\sigma|\xi - \eta|}e^{\sigma|\eta|}, \\
\langle \xi \rangle^{s} & \lesssim \langle \xi -\eta \rangle^{s} + \langle \eta \rangle^{s}.
\end{aligned}\]
It follows from these observations that
\[\begin{aligned}
\| uv \|_{G^{\sigma,s}} & \lesssim \left\| \int_{\R} \left[e^{\sigma|\xi - \eta|}\langle \xi -\eta \rangle^{s}|\hat{u}(\xi - \eta)|\right] \left[e^{\sigma|\eta|}|\hat{v}(\eta)|\right] \ d\eta \right\|_{L^{2}_{\xi}} \\
& \quad + \left\| \int_{\R} \left[e^{\sigma|\xi - \eta|}|\hat{u}(\xi - \eta)|\right] \left[e^{\sigma|\eta|}\langle \eta \rangle^{s}|\hat{v}(\eta)|\right] \ d\eta \right\|_{L^{2}_{\xi}}.
\end{aligned}\]
Applying Young's inequality to this, we obtain
\[\begin{aligned}
\| uv \|_{G^{\sigma,s}} & \lesssim \left\| e^{\sigma|\xi|}\langle \xi \rangle^{s} |\hat{u}(\xi)| \right\|_{L^{2}_{\xi}} \left\| e^{\sigma|\xi|}|\hat{v}(\xi)| \right\|_{L^{1}_{\xi}} \\
& \quad + \left\| e^{\sigma|\xi|}|\hat{u}(\xi)| \right\|_{L^{1}_{\xi}} \left\| e^{\sigma|\xi|}\langle \xi \rangle^{s} |\hat{v}(\xi)| \right\|_{L^{2}_{\xi}} \\
& \lesssim \| u \|_{G^{\sigma,s}} \| v \|_{G^{\sigma,s}}.
\end{aligned}\]
Here, we have used the fact that
\[
\int_{\R} f\ dx \leq \left( \int_{\R} \langle x \rangle^{-2s}\ dx \right)^{1/2} \left( \int_{\R} \langle x \rangle^{2s} |f(x)|^2 \ dx\right)^{1/2},
\]
and that the first integral on the right converges for $s > 1/2$.
\end{proof}
With all these facts in mind, we will use the following strategy to prove Theorem \ref{main}:
\begin{itemize}
\item The assumptions on $f$ imply that $f \in G^{\sigma_0}$.  By the embedding in equation \eqref{embedding1}, $f \in G^{\sigma', s'}$ for any $\sigma' < \sigma_{0}$ and $s' \in \mathbb{R}$.  We use this fact to construct local solutions in $G^{\sigma', s'}$ for $s' > 1/2$.

\item By a standard argument, it suffices to show that the $G^{\sigma', s'}$ norm of the solution $u$ remains finite in the interval $[0,T]$ for the solution to exist up to time $T > 0$.  As will be shown in section \ref{GWP}, this will require that we choose $\sigma'$ sufficiently small, and $s' = 1$.

\item Once it is known that the $G^{\sigma', 1}$ norm remains finite, the embedding \eqref{embedding2} will imply that the $L^{2}$ norm remains bounded up to time $T$.  Thus, by the standard $L^{2}$ theory, the solutions may be continued up to time $T$ in $L^{2}$.  Moreover since $u(t) \in G^{\sigma', 1}$, it will also be analytic.

\end{itemize}


\section{Local Well-Posedness}\label{LWP}
To begin, let us first recall some basic facts about the Schr\"odinger equation.  Recall that the Cauchy problem
\[\begin{aligned}
& iu_{t} + \Delta u = F \\
& \ \ u(x,0) = f
\end{aligned}\]
can be rewritten in integral form using the Duhamel formula
\[
u(x,t) = e^{it\Delta}f - i \int_{0}^{t} e^{i(t-\tau)\Delta} F(\tau)\ d\tau.
\]
Applying this to equation \eqref{NLSCauchy}, we have
\begin{equation}\label{duhamel}
u(x,t) = e^{it\Delta}f - i \int_{0}^{t} e^{i(t-\tau)\Delta} |u(\tau)|^{p-1}u(\tau)\ d\tau.
\end{equation}
A \emph{strong} solution to \eqref{NLSCauchy} is a solution to the integral equation \eqref{duhamel}.  With this definition in mind, we may state our local result in a precise form.
\begin{prop}\label{LWPprop}
Let $p$ be an odd natural number, $\sigma' \geq 0$, and let $s' > 1/2$.  Then the Cauchy problem \eqref{NLSCauchy} is locally well-posed in $G^{\sigma', s'}(\R)$.  That is, for any $f\in G^{\sigma', s'}$, there exists a time $\delta = \delta(\|f\|) > 0$ such that the Cauchy problem \eqref{NLSCauchy} has a unique strong solution
\[
u \in C\left([0,\delta); G^{\sigma', s'}\right).
\]
Furthermore, the solution map $f \mapsto u$ is Lipschitz continuous from $G^{\sigma', s'}$ to $C\left([0,\delta); G^{\sigma', s'}\right)$.
\end{prop}

\begin{proof}
Fix $f \in G^{\sigma',s'}$, and define an operator $\Phi$ on $G^{\sigma',s'}$ by 
\[
\Phi(u)=e^{it\Delta}f-i\int_0^t e^{i(t-\tau)\Delta}|u(x,\tau)|^{p-1} u(x, \tau)\,d\tau.
\]
Since the operator $e^{i t \Delta}$ is unitary, it is easy to see that this integral formula implies the inequality
\begin{equation}\label{energyineq}
\| \Phi(u)(t) \|_{G^{\sigma',s'}} \leq \| f \|_{G^{\sigma',s'}} + \int_{0}^{t} \| |u(x,\tau)|^{p-1} u(x, \tau) \|_{G^{\sigma',s'}}\ d\tau.
\end{equation}
By Lemma \ref{Lemma3} we have, for $s' > 1/2$, 
\[
\|\Phi(u) (t) \|_{G^{\sigma',s'}} \leq \|f\|_{G^{\sigma',s'}} + C\int_{0}^{t}\| u(\tau) \|_{G^{\sigma',s'}}^{p} d\tau
\]
for some generic constant $C > 0$.  Taking the supremum over $t \in [0,\delta)$ gives us
\begin{equation}\label{maxnorm}
\|\Phi(u) \|_{L^{\infty}G^{\sigma',s'}} \leq \| f \|_{G^{\sigma',s'}} + C \delta \| u \|_{L^{\infty}G^{\sigma,s'}}^{p}.
\end{equation}
It follows that $\Phi$ maps $C([0,\delta); G^{\sigma',s'})$ to itself.

Next, we show that $\Phi$ is a contraction.  The existence of a unique fixed point will follow from the Contraction Mapping Principle.  Let $u, v \in C([0,\delta); G^{\sigma',s'})$ such that
\[
\| u \|_{L^{\infty}G^{\sigma',s'}} \leq \| f \|_{G^{\sigma',s'}} \textrm{ and } \| v \|_{L^{\infty}G^{\sigma',s'}} \leq \| f \|_{G^{\sigma',s'}}.  
\]
By applying equation \eqref{energyineq}, it is easy to see that
\[\begin{aligned}
\|\Phi(u)-\Phi(v)\|_{L^{\infty} G^{\sigma',s'}} & \leq C \delta \left(\| u \|_{L^{\infty}G^{\sigma',s'}}^{p-1} + \| v \|_{L^{\infty}G^{\sigma',s'}}^{p-1} \right) \|u-v\|_{L^{\infty} G^{\sigma',s'}} \\
& \leq 2C \delta \| f \|_{G^{\sigma',s'}}^{p-1} \| u - v \|_{L^{\infty}G^{\sigma',s'}}.
\end{aligned}\]
If
\[
\delta < \frac{1}{2 C \| f \|_{G^{\sigma',s'}}^{p-1}}
\]
then
\[
\|\Phi(u)-\Phi(v)\|_{L^{\infty}G^{\sigma',s'}} < \|u-v\|_{L^{\infty}G^{\sigma',s'}}.
\]
Thus $\Phi$ is a contraction.  The existence of a unique fixed point for $u$ follows from the Contraction Mapping Principle.  This fixed point satisfies equation \eqref{duhamel}, and so is a strong solution to the Cauchy problem \eqref{NLSCauchy}.

Finally, we must show that the solution map $f \mapsto u$ is continuous from $G^{\sigma', s'}$ to $L^{\infty}G^{\sigma', s'}$.  Suppose for two initial conditions $f, g\in G^{\sigma', s'}$, respectively with
\[
\| f \|_{G^{\sigma', s'}} \leq R \quad \textrm{and} \quad \| g \|_{G^{\sigma', s'}} \leq R
\]
we have the corresponding solutions $u$ and $v$, respectively.  As in the computations above, we may apply equation \eqref{energyineq} and Lemma \ref{Lemma3} to obtain
\[
\begin{aligned}
\|u-v\|_{L^{\infty}G^{\sigma', s'}} & \leq \|f-g\|_{G^{\sigma', s'}} + \int_{0}^{t}\||u|^{p-1}u-|v|^{p-1}v\|_{G^{\sigma', s'}} d\tau \\
& \leq \|f-g\|_{G^{\sigma', s'}} \\
& \quad + C\delta \left(\|u\|_{L^{\infty}G^{\sigma', s'}}^{p-1} + \|v\|_{L^{\infty}G^{\sigma', s'}}^{p-1}\right) \|u-v\|_{L^{\infty}G^{\sigma', s'}}.
\end{aligned}
\]
From equation \eqref{maxnorm} and the choice of $\delta$, it follows that
\[
\|u-v\|_{L^{\infty}G^{\sigma', s'}} \leq \|f-g\|_{G^{\sigma', s'}} + 2 C \delta R^{p-1} \|u-v\|_{L^{\infty}G^{\sigma', s'}}.
\]
If we now further make the assumption that $\delta$ also satisfies
\[
2 C \delta R^{p-1} < 1,
\]
then we may conclude that
\[
\|u-v\|_{L^{\infty}G^{\sigma', s'}} \leq \frac{\|f-g\|_{G^{\sigma', s'}}}{1-2C\delta R^{p-1}}.
\]
Continuity of the solution map follows.  This completes the proof of Proposition \ref{LWPprop}.

\end{proof}


\section{Global Existence}\label{GWP}

In this section, we will prove the following proposition, from which the second conclusion of Theorem \ref{main} will follow:
\begin{prop}\label{propglobal}
Let $u$ be the local solution to the Cauchy problem \eqref{NLSCauchy}, and let $T > 0$ be arbitrary.  Then there exists $\sigma > 0$ such that
\[
\sup_{t \in [0,T]} \| u(\cdot,t) \|_{G^{\sigma, 1}(\R)} < C
\]
for some constant $C > 0$.
\end{prop}
\noindent To prove this, we first state a preliminary lemma.
\begin{lem}\label{lemest}
Let $n \in \mathbb{N}$, $n \geq 2$, and let $\eta_{1}, \ldots, \eta_{n} \in \mathbb{R}$. Then
\[
e^{\sigma \sum_{j = 1}^{n}|\eta_{j}|} - e^{\sigma \left|\sum_{j = 1}^{n}\eta_{j}\right|} \leq \sum_{k = 1}^{n}\left(2 \sigma \min\left( \left| \sum_{j \neq k} \eta_{j} \right|, |\eta_{k}| \right)\right)^{\theta}e^{\sigma \sum_{j = 1}^{n}|\eta_{j}|}.
\]
for any $\theta \in [0,1]$.
\end{lem}
\begin{proof}
By strong induction on $n$.  The case $n = 2$ was shown in \cite{SD2016}.  Thus, we may assume the result holds for $n \leq m$.  Consider the case $n = m+1$.  We may write
\[\begin{aligned}
e^{\sigma \sum_{j = 1}^{m+1}|\eta_{j}|} - e^{\sigma \left|\sum_{j = 1}^{m+1}\eta_{j}\right|} & = e^{\sigma|\eta_{m+1}|}\left[ e^{\sigma\sum_{j = 1}^{m}|\eta_{j}|} - e^{\sigma|\sum_{j = 1}^{m}\eta_{j}|}\right] \\
& \quad + e^{\sigma|\eta_{m+1}|}e^{\sigma|\sum_{j = 1}^{m}\eta_{j}|} - e^{\sigma \left|\sum_{j = 1}^{m+1}\eta_{j}\right|}.
\end{aligned}\]
Applying the inductive hypothesis to the first line above, we have
\[
e^{\sigma\sum_{j = 1}^{m}|\eta_{j}|} - e^{\sigma|\sum_{j = 1}^{m}\eta_{j}|} \leq \sum_{k = 1}^{m}\left(2 \sigma \min\left( \left| \sum_{j \neq k} \eta_{j} \right|, |\eta_{k}| \right)\right)^{\theta}e^{\sigma\sum_{j = 1}^{m}|\eta_{j}|}.
\]
Applying the inductive hypothesis and the triangle inequality to the second line, we have
\[\begin{aligned}
& e^{\sigma|\eta_{m+1}|}e^{\sigma|\sum_{j = 1}^{m}\eta_{j}|} - e^{\sigma \left|\sum_{j = 1}^{m+1}\eta_{j}\right|} \\
& \qquad \leq \left( 2 \sigma \min \left( \left|\sum_{j = 1}^{m} \eta_{j}  \right|, |\eta_{m+1}|\right) \right)^{\theta}e^{\sigma\sum_{j = 1}^{m+1}|\eta_{j}|}.
\end{aligned}\]
The desired result follows.
\end{proof}

\begin{proof}[Proof of Proposition \ref{propglobal}]
Recall that 
\[
\| f \|_{H^{s}} \sim \| f \|_{H^{s-1}} + \| \nabla f \|_{H^{s-1}}
\]
for $f \in H^{s}$ (see \cite{T2006}, Appendix A).  By the commutativity of Fourier multipliers, this implies that
\[
\| u \|_{G^{\sigma,1}} \sim \| u \|_{G^{\sigma}} + \| \nabla u \|_{G^{\sigma}} \sim \left( \| u \|_{G^{\sigma}}^2 + \| \nabla u \|_{G^{\sigma}}^2 \right)^{1/2}
\]
for $u \in G^{\sigma,1}$.  Thus, it suffices to estimate the norms above to obtain the desired result.

Let $\Lambda$ be the pseudodifferential operator defined by the Fourier multiplier
\[
\widehat{\Lambda u} = e^{\sigma|\xi|} \hat{u}(\xi,t),
\]
and define
\[
U(x,t) = \Lambda u.
\]
We observe that
\[
\| u \|_{G^{\sigma}} = \| U \|_{L^{2}} \quad \textrm{and} \quad \| \nabla u \|_{G^{\sigma}} = \| \nabla U \|_{L^{2}}.
\]
Moreover, it is easy to see that $U$ and $\overline{U}$ satisfy the equations
\[
U_{t} = i \Delta U - i \Lambda\left(| u |^{p-1}u\right)
\]
and
\[
\overline{U}_{t} = -i \Delta \overline{U} + i \Lambda \left(| u |^{p-1}\overline{u}\right).
\]
Next, define a quantity $S(t)$ by
\[
S(t) = \int_{\R} | U(x,t) |^{2} + |\nabla U|^2 + \frac{2}{p+1}|U|^{p+1}\ dx.
\]
Observe that in the case $\sigma = 0$, this quantity would be conserved for equation \eqref{NLS}.  By the Fundamental Theorem of Calculus, we have that
\[
S(t) = S(0) + \int_{0}^{t} \frac{dS(\tau)}{dt} \ d\tau.
\]
A lengthy computation will show that
\[\begin{aligned}
\frac{dS}{dt} & = i\int_{\R} \overline{\nabla U} \cdot \nabla N(u) - \nabla N(\overline{u}) \cdot \nabla U\ dx \\
& \quad + i\int_{\R} \overline{U} N(u) - N(\overline{u}) U\ dx \\
& \quad + i\int_{\R} |U|^{p-1}\overline{U}N(u) - N(\overline{u})|U|^{p-1}U\ dx,
\end{aligned}\]
where
\[
N(u) = |U|^{p-1}U - \Lambda(|u|^{p-1}u).
\]
Applying H\"older's inequality, we see that
\begin{equation}\label{holder}\begin{aligned}
& S(t) \leq S(0) + 2\int_{0}^{t} \| \nabla U(\tau) \|_{L^{2}}\|\nabla N(u)(\tau) \|_{L^{2}}\ d\tau \\
& \qquad \qquad + 2\int_{0}^{t} \| U(\tau) \|_{L^{2}}\| N(u)(\tau) \|_{L^{2}}\ d\tau \\
& \qquad \qquad + 2 \int_{0}^{t} \left\| |U(\tau)|^{p-1}U(\tau)  \right\|_{L^{2}} \| N(u(\tau)) \|_{L^{2}}\ d\tau.
\end{aligned}\end{equation}
Note that
\[
\left\| |U(\tau)|^{p-1}U(\tau)  \right\|_{L^{2}} = \left\| U(\tau)  \right\|_{L^{2p}}^{p}.
\]
We estimate this using the Gagliardo-Nirenberg inequality, giving us
\[
\left\| U(\tau)  \right\|_{L^{2p}} \lesssim \| U(\tau) \|_{L^{2}}^{1-\alpha} \| \nabla U \|_{L^{2}}^{\alpha}
\]
with
\[
\alpha = \frac{1}{2}\left( 1 - \frac{1}{p} \right).
\]
Thus, equation \eqref{holder} becomes
\begin{equation}\begin{aligned}\label{holder2}
S(t) & \leq S(0) + 2\int_{0}^{t} \| U(\tau) \|_{L^{2}}\| N(u)(\tau) \|_{L^{2}}\ d\tau \\
& \quad + 2\int_{0}^{t} \| \nabla U(\tau) \|_{L^{2}}\|\nabla N(u)(\tau) \|_{L^{2}}\ d\tau \\
& \quad + 2C \int_{0}^{t} \left( \| U(\tau) \|_{L^{2}}^{1-\alpha} \| \nabla U \|_{L^{2}}^{\alpha} \right)^{p} \| N(u(\tau)) \|_{L^{2}}\ d\tau.
\end{aligned}\end{equation}

Next, we must estimate each of the terms involving the nonlinear operator $N(u)$.  By Plancherel's theorem, it suffices to consider
\[
\| \widehat{N(u)} \|_{L^{2}_{\xi}} \quad \textrm{and} \quad \| \widehat{\nabla N(u)} \|_{L^{2}_{\xi}}.
\]
For this, we first rewrite $N(u)$ as
\[
N(u) = |e^{\sigma|\nabla|}u|^{p-1}(e^{\sigma|\nabla|}u) - \Lambda\left( |u|^{p-1}u \right).
\]
We then take the spatial Fourier transform, which we write as the convolution integral
\[
\widehat{N(u)} = \int_{H}\left(e^{\sum_{j=1}^{2k+1} \sigma|\eta_j|} - e^{\sigma|\xi|} \right)\left[\prod_{j=1}^{k} \hat{\overline{u}}(\eta_{2j-1}) \hat{u}(\eta_{2j})\right] \hat{u}(\eta_{2k+1}) \ d\eta,
\]
where $H$ denotes the hyperplane $\xi = \eta_{1} + \cdots + \eta_{2k+1}$ and
\[
d\eta = d\eta_1 \cdots d \eta_{2k+1}.
\]
To estimate $|\widehat{N(u)}|$, we apply Lemma \ref{lemest} to obtain
\begin{equation}\label{nestimate}\begin{aligned}
|\widehat{N(u)}| & \leq C \sigma^{\theta} \sum_{m = 1}^{2k+1} \int_{H} \left( \min \left\{ \left| \sum_{m \neq j} \eta_{j} \right|, |\eta_{m}| \right\} \right)^{\theta} G(\eta)\ d\eta \\
& \leq C \sigma^{\theta} \sum_{m = 1}^{2k+1} \int_{H} |\eta_{m}|^{\theta} G(\eta)\ d\eta,
\end{aligned}\end{equation}
where $C > 0$ is a generic constant which may be different in each line, $\eta = (\eta_{1}, \ldots, \eta_{2k+1})$, and
\[\begin{aligned}
G(\eta) & = \left[\prod_{j=1}^{k} \left|e^{\sigma|\eta_{2j-1}|}\hat{\overline{u}}(\eta_{2j-1}) \right| \left| e^{\sigma|\eta_{2j}|} \hat{u}(\eta_{2j-1})\right|\right] \left|e^{\sigma|\eta_{2k+1}}\hat{u}(\eta_{2k+1})\right| \\
& = \left[\prod_{j=1}^{k} \left|\hat{\overline{U}}(\eta_{2j-1}) \right| \left| \hat{U}(\eta_{2j-1})\right|\right] \left|\hat{U}(\eta_{2k+1})\right|.
\end{aligned}\]
If we define $\hat{v}_{j} = |\hat{U}(\eta_{j})|$ or $|\hat{\overline{U}}(\eta_{j})|$, as appropriate, then we can rewrite each of the integrals in equation \eqref{nestimate} as a convolution of the form
\[
\hat{v}_{1} * \ldots * \hat{v}_{j-1} * \widehat{|\nabla|^{\theta}v_{j}} * \hat{v}_{j+1} * \cdots * \hat{v}_{2k+1}.
\]
Thus we have that
\[\begin{aligned}
\| \widehat{N(u)} \|_{L^{2}} & \lesssim \sigma^{\theta} \sum_{m = 1}^{2k+1} \| \hat{v}_{1} * \ldots * \hat{v}_{m-1} * \widehat{|\nabla|^{\theta}v_{m}} * \hat{v}_{m+1} * \cdots * \hat{v}_{2k+1} \|_{L^{2}} \\
& \lesssim \sigma^{\theta} \sum_{m = 1}^{2k+1} \| v_{1} \ldots v_{m-1} (|\nabla|^{\theta} v_{m}) v_{m+1} \cdots v_{2k+1} \|_{L^{2}} \\
& \lesssim \sigma^{\theta} \sum_{m = 1}^{2k+1} \left(\prod_{j \neq m} \| v_{j} \|_{L^{\infty}}\right) \| |\nabla|^{\theta} v_{m} \|_{L^{2}} \\
& \lesssim \sigma^{\theta} \sum_{m = 1}^{2k+1} \left(\prod_{j \neq m} \| v_{j} \|_{H^{1}}\right) \| v_{m} \|_{H^{1}} \\
& \lesssim \sigma^{\theta} \prod_{m = 1}^{2k+1} \| v_{j} \|_{H^{1}} \\
& \lesssim \sigma^{\theta} \left(\| U \|_{L^{2}}^{2} + \| \nabla U \|_{L^{2}}^{2} \right)^{p/2}.
\end{aligned}\]
It follows that
\begin{equation}\label{finalest1}
\| N(u) \|_{L^{2}} \lesssim \sigma^{\theta} \left(\| U \|_{L^{2}}^{2} + \| \nabla U \|_{L^{2}}^{2} \right)^{p/2}.
\end{equation}

For the $\nabla N(u)$ terms, we observe that $|\widehat{\nabla N(u)}| = |\xi| |\widehat{ N(u)}|$. Analogously, the problem of estimating the norm of $\widehat{\nabla N(u)}$ can be reduced by Lemma \ref{lemest} to estimating a sum of terms of the form
\[\begin{aligned}
& \sigma^{\theta} \int_{H} |\xi|\left( \min \left\{\left|\sum_{j \neq \ell}\eta_{j}\right|, |\eta_{\ell}| \right\} \right)^{\theta} \times \\
& \qquad \times \left[\prod_{j=1}^{k} |e^{\sigma|\eta_{2j}|} \hat{\overline{u}}(\eta_{2j}) | |e^{\sigma|\eta_{2j-1}|}\hat{u}(\eta_{2j-1})|\right] |e^{\sigma|\eta_{2k+1}|}\hat{u}(\eta_{2k+1})|\ d\eta.
\end{aligned}\]
To estimate these integrals, we recall that $\xi = \eta_{1} + \cdots + \eta_{2k+1}$, so that
\[
|\xi| \leq \left|\sum_{j \neq \ell} \eta_{j}\right| + |\eta_{\ell}|.
\]
In the case where
\[
\left|\sum_{j \neq \ell} \eta_{j}\right| \geq |\eta_{\ell}|,
\]
we obtain that
\[
|\xi| \leq 2\left|\sum_{j \neq \ell} \eta_{j}\right| \leq 2\sum_{j \neq \ell} \left|\eta_{j}\right|.
\]
It follows that $|\widehat{\nabla N(u)}|$ can be estimated by a sum of terms of the form
\[
\sigma^{\theta} \int_{H} |\eta_{\ell}|^{\theta} |\eta_{n}| \left[\prod_{j=1}^{k} | \hat{\overline{U}}(\eta_{2j}) | |\hat{U}(\eta_{2j-1})|\right] |\hat{U}(\eta_{2k+1})|\ d\eta,
\]
where $\ell \neq n$.  The case $\left|\sum_{j \neq \ell} \eta_{j}\right| \leq |\eta_{\ell}|$ is similar.  For both of these cases, we observe that these integrals can be written in convolution form as
\[
\hat{v}_{1} * \cdots * \hat{v}_{j-1} * \left( \widehat{|\nabla|^{\theta} v_{j}} \right) * \hat{v}_{j+1} * \cdots * v_{2k} * (\widehat{\nabla v_{2k+1}}).
\]
We estimate these terms by
\[\begin{aligned}
& \| v_{1} \cdots v_{j-1} \left(|\nabla|^{\theta} v_{j}\right) v_{j+1}  \cdots  v_{2k} \nabla v_{2k+1} \|_{L^{2}_{x}} \\
& \qquad \qquad \qquad \qquad \qquad \lesssim \| |\nabla| v_{2k+1}\|_{L^{2}} \| |\nabla|^{\theta} v_{j} \|_{L^{\infty}} \prod_{\ell \neq j, 2k+1} \| v_{\ell} \|_{L^{\infty}} \\
& \qquad \qquad \qquad \qquad \qquad \lesssim \prod_{j = 1}^{2k+1} \| v_{j} \|_{H^{1}} \\
& \qquad \qquad \qquad \qquad \qquad \lesssim \left( \| U \|_{L^{2}}^{2} + \| \nabla U \|_{L^{2}}^{2} \right)^{p/2}.
\end{aligned}\]
We remark that we have once again used the Gagliardo-Nirenberg inequality to estimate
\[
\| |\nabla|^{\theta} v_{j} \|_{L^{\infty}} \lesssim \| |\nabla|^{\theta} v_{j} \|_{L^{2}}^{\beta} \| |\nabla|^{\theta} v_{j} \|_{\dot{H}^{1-\theta}}^{1-\beta} \lesssim \| v_{j} \|_{H^{1}},
\]
which requires that 
\[
\frac{1}{2} = \beta(1-\theta).
\]
Thus, it is necessary that $0 \leq \theta < 1$.  We may now conclude that
\begin{equation}\label{finalest2}
\| \nabla N(u) \|_{L^{2}} \lesssim \sigma^{\theta} \left( \| U \|_{L^{2}}^{2} + \| \nabla U \|_{L^{2}}^{2} \right)^{p/2}
\end{equation}
If we now combine equations \eqref{holder2}, \eqref{finalest1}, and \eqref{finalest2}, we obtain that
\begin{equation}\label{sestimate}
S(t) \leq S(0) + C \sigma^{\theta}\int_{0}^{t} S^{\frac{p}{2}}(\tau)\left( 2S^{1/2}(\tau) + S^{p/2}(\tau) \right)\ d\tau.
\end{equation}

Our next step is to show that this quantity remains bounded for $t \in [0,T]$.  We apply a simple bootstrap argument.  Let H$(t)$ and C$(t)$ be the statements
\begin{itemize}
\item H$(t)$: $S(\tau) \leq 4 S(0)$ for $0 \leq \tau \leq t$.
\item C$(t)$: $S(\tau) \leq 2 S(0)$ for $0 \leq \tau \leq t$.
\end{itemize}
To close the bootstrap, we must prove the following four statements:
\begin{enumerate}
\item[(a)] H$(t) \Rightarrow$ C$(t)$;
\item[(b)] C$(t) \Rightarrow$ H$(t')$ for all $t'$ in a neighborhood of $t$;
\item[(c)] If $t_1, t_2, \ldots$ is a sequence in $[0,T]$ such that $t_n \rightarrow t \in [0,T]$, with C$(t_{n})$ true for all $t_n$, then C$(t)$ is also true;
\item[(d)] H$(t)$ is true for at least one $t \in [0,T]$.
\end{enumerate}

\begin{proof}[Proof of (a)]
Assuming the statement H$(t)$, equation \eqref{sestimate} gives us the estimate
\[
S(t) \leq S(0) + C\sigma^{\theta}(4S(0))^{\frac{p}{2}}\left( 2(4S(0))^{1/2} + (4S(0))^{p/2} \right)t.
\]
Taking the supremum, this gives us
\[
\sup_{t \in [0,T]} S(t) \leq S(0) + C\sigma^{\theta}(4S(0))^{\frac{p}{2}}\left( 2(4S(0))^{1/2} + (4S(0))^{p/2} \right)T.
\]
Choose $\sigma$ so that
\[
\sigma \leq \left[ C (4S(0))^{\frac{p}{2}} \left( 2(4S(0))^{1/2} + (4S(0))^{p/2} \right) T \right]^{-1/\theta},
\]
then the conclusion C$(t)$ follows.  Note that this requires that $\theta > 0$.
\end{proof}

\begin{proof}[Proof of (b)]
Assume $S(\tau) \leq 2 S(0)$ for all $\tau$ with $0 \leq \tau \leq t$.  We may then apply the local existence theory to construct solutions which exist on an interval $[t,t+\delta) \subset [0,T]$ for some small $\delta > 0$.  In particular, we can do this so that
\[
\sup_{\tau \in [t,t+\delta)} S(\tau) \leq 4 S(0).
\]
Since
\[
\sup_{\tau \in (t-\delta,t]} S(\tau) \leq 2 S(0)
\]
by assumption, the statement H$(t')$ holds for all $t' \in (t-\delta, t + \delta)$.
\end{proof}

\begin{proof}[Proof of (c)]
The statement in (c) follows immediately from the fact that our solutions are constructed so that the norm $\| u(t) \|_{G^{\sigma, 1}}$ defines a continuous function in time.
\end{proof}

\begin{proof}[Proof of (d)]
H$(0)$ holds by assumption.
\end{proof}

Based on the above, we may close the bootstrap, and it follows that C$(t)$ holds for all $t \in [0,T]$.  
\end{proof}
To conclude the proof of Theorem \ref{main}, we summarize what we have accomplished: our assumptions on $f$ imply that $f \in G^{\sigma_{0}}$.  By the Gevrey embedding \eqref{embedding1}, $f \in G^{\sigma,1}$ for $\sigma$ given by
\[
\sigma < \min \left\{ \sigma_{0}, \left[ C (4S(0))^{\frac{p}{2}} \left( 2(4S(0))^{1/2} + (4S(0))^{p/2} \right) T \right]^{-1-\epsilon}\right\}.
\]
Using the local theory of section \ref{LWP}, we can construct a solution $u$ up to some small time $\delta > 0$.  By the global theory of section \ref{GWP}, the norm of $u$ remains bounded, so we may continue our solution past time $\delta$ to the desired time $T$.
\section{Acknowledgments}

The authors would like to thank Alejandro J. Castro, whose comments provided the inspiration for this work.

\bibliographystyle{amsplain}
\bibliography{database}
\end{document}